\documentclass[12pt]{article}
\usepackage[margin=1in]{geometry}
\usepackage{url}
\usepackage{graphicx}
\usepackage{subcaption}

\usepackage[T1]{fontenc}
\usepackage{fouriernc}

\usepackage{amsmath}
\usepackage{amssymb}
\usepackage{amsthm}
\usepackage{mathrsfs}

\usepackage{pgf,tikz}
\usetikzlibrary{arrows, automata, shapes, positioning}
\usetikzlibrary{decorations}
\usetikzlibrary{snakes}
\usepackage{verbatim}
\usepackage{xcolor} 

\theoremstyle{plain}
\newtheorem{theorem}{Theorem}

\theoremstyle{definition}

\theoremstyle{remark}

\newcommand{\VTM}{\mbox{\hspace{.01in}{\tt vtm}}}

\title{On some problems of Harju concerning squarefree arithmetic
  progressions in infinite words}
\author{
James Currie and Narad Rampersad\thanks{The authors are supported by
  NSERC Discovery Grants 03901-2017 (Currie) and 418646-2012 (Rampersad).}\\
Department of Mathematics and Statistics\\
University of Winnipeg\\
\url{{j.currie, n.rampersad}@uwinnipeg.ca}}
\date{\today}

\begin{document}
\maketitle

\begin{abstract}
  In a recent paper, Harju posed three open problems concerning
  square-free arithmetic progressions in infinite words.  In this note
  we solve two of them.
\end{abstract}

\section{Introduction}
The study of infinite words avoiding squares is a classical problem in
combinatorics on words.  A \emph{square} is a word of the form $xx$,
such as \texttt{tartar}.  One of the most well-studied squarefree
words \cite{Thu12, Hal64} is the word $\VTM = 012021012102012\cdots$
obtained by iterating the map $0 \mapsto 012$, $1 \mapsto 02$, $2
\mapsto 1$.

Harju \cite{Har18} studied the following question and showed that it
has a positive solution for all $k \geq 3$:
\begin{quote}
  Given $k$, does there exist an infinite squarefree sequence
  $(w_n)_{n \geq 0}$ over a ternary alphabet such that the subsequence
  $(w_{kn})_{n \geq 0}$ is squarefree?
\end{quote}
Carpi \cite{Car88}, Currie and Simpson \cite{CS02}, and Kao et
al.~\cite{KRSS08} also studied similar problems.  Harju ended his
paper with three open problems:
\begin{enumerate}
  \item Does there exist a squarefree sequence $(w_n)_{n \geq 0}$ over
    a ternary alphabet such that for every $k \geq 3$, the subsequence
    $(w_{kn})_{n \geq 0}$ contains a square?
  \item Do there exist pairs $(k,\ell)$ of relatively prime integers
    such that there exists a squarefree sequence $(w_n)_{n \geq 0}$
    over a ternary alphabet for which both $(w_{kn})_{n \geq
      0}$ and $(w_{\ell n})_{n \geq 0}$ are squarefree?
  \item It is true that for all squarefree words $(w_n)_{n \geq 0}$
    over a ternary alphabet, there exists a word $(v_n)_{n \geq 0}$
    and an integer $k \geq 3$ such that $(v_{kn})_{n \geq 0} =
    (w_n)_{n \geq 0}$?
\end{enumerate}
In this note we show that the word $\VTM$ gives a positive answer to
the first problem.  We also show that there a positive answer to the
third problem for every $k \geq 23$.

\section{The main results}
We recall that $\VTM = (v_n)_{n \geq 0}$ is the fixed point (starting with
$0$) of the morphism that maps $0 \mapsto 012$, $1 \mapsto 02$, $2
\mapsto 1$.

\begin{theorem}
  For each $k \geq 2$ the sequence $(v_{kn})_{n \geq 0}$ contains
  either the square $00$ or the square $22$.
\end{theorem}

\begin{proof}
  The first part of the proof relies on the fact that $\VTM$ is a
  $2$-automatic sequence.  Berstel~\cite{Ber79} studied several
  different ways to generate the sequence $\VTM$; in particular, he
  showed that $\VTM$ is generated by the $2$-DFAO (\emph{deterministic
    finite automaton with output}) in Figure~\ref{vtm_aut}.  The
  automaton takes the binary representation of $n$ as input, and if
  the computation ends in a state labeled $a$, the automaton outputs
  $a$, indicating that $v_n = a$.

  \begin{figure}[htb]
    \centering
    \includegraphics[scale=0.75]{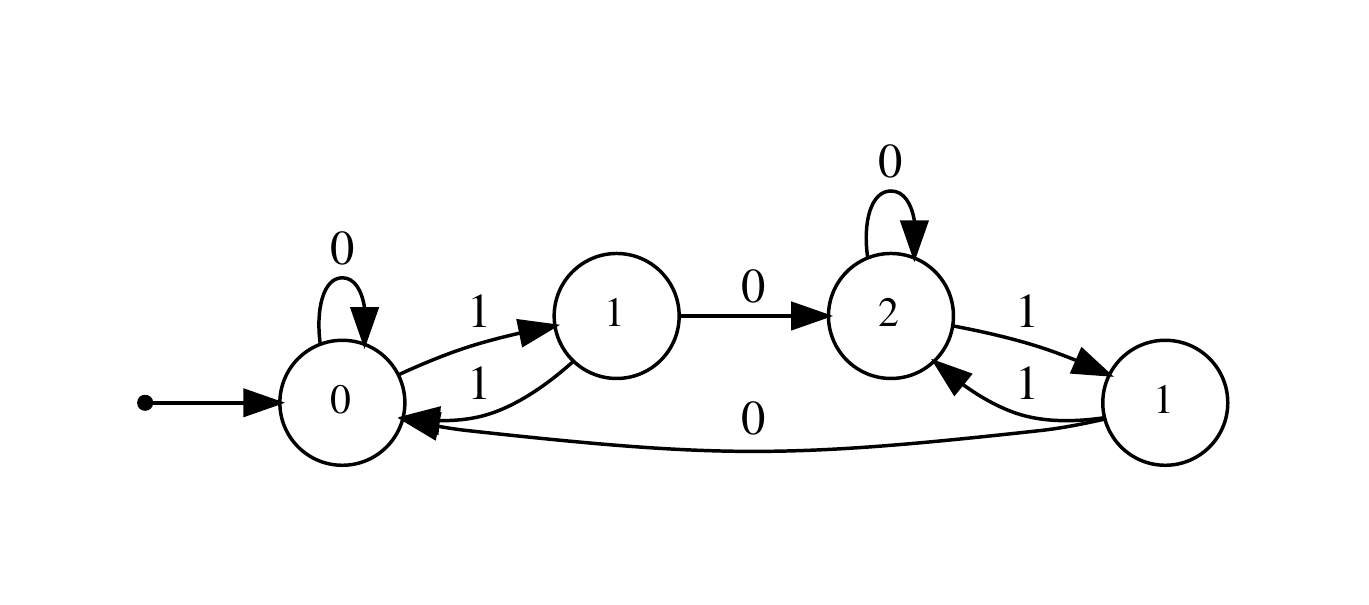}
    \caption{$2$-DFAO for $\VTM$}\label{vtm_aut}
  \end{figure}

  Since $\VTM$ is an automatic sequence, we can use
  Walnut~\cite{Walnut} to verify that it has certain combinatorial
  properties.  We verify with Walnut that for every $k \geq 2$, the
  sequence $\VTM$ contains a length $k+1$ factor of the form $0u0$ or
  $2u2$.  The Walnut command to do this is:
  \begin{center}
  \texttt{eval same\_first\_last ``Ei (VTM[i]=@0 \&
    VTM[i+k]=@0)|(VTM[i]=@2 \& VTM[i+k]=@2)'';}
  \end{center}
  The Walnut output for this command is the automaton in
  Figure~\ref{walnut_aut}, which shows that the given predicate holds
  for all $k \geq 2$.

  \begin{figure}[htb]
    \centering
    \includegraphics[scale=0.75]{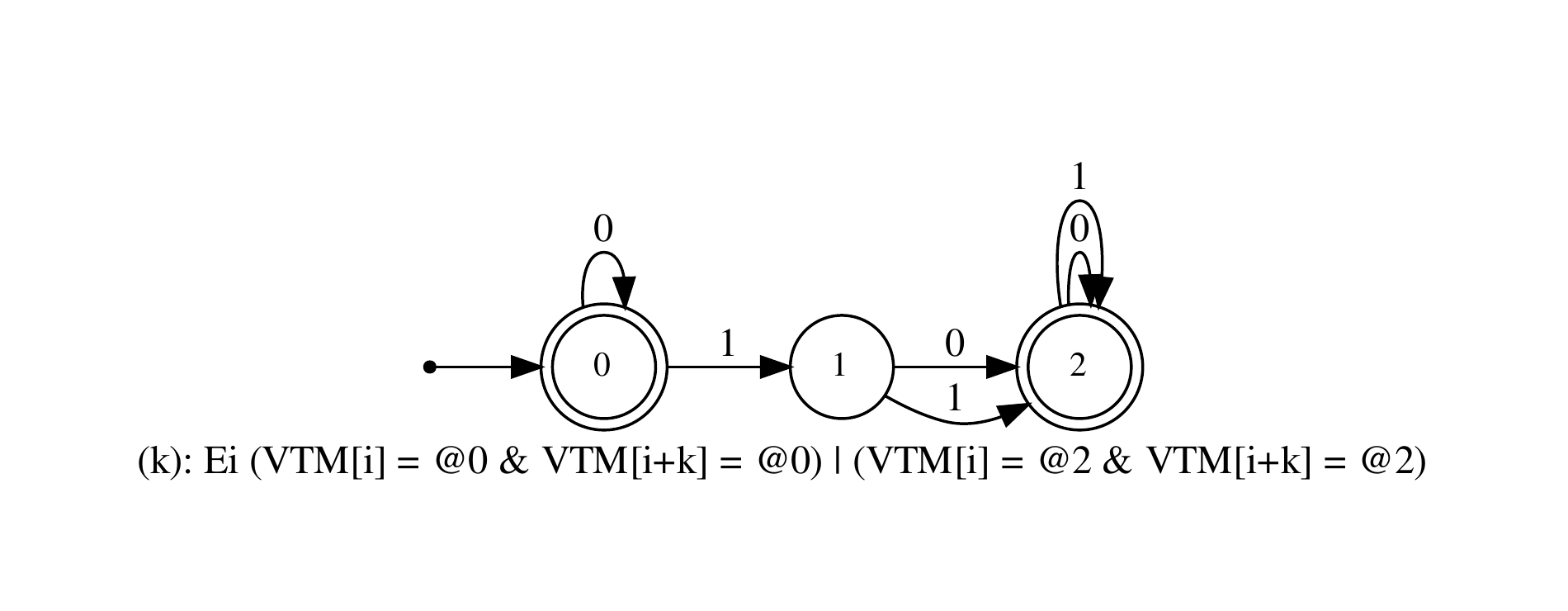}
    \caption{Walnut output automaton}\label{walnut_aut}
  \end{figure}

  To complete the proof, it suffices to show that $\VTM$ contains an
  occurrence of the length $k+1$ factor $0u0$ or $2u2$ at a position
  congruent to $0$ modulo $k$.  Blanchet-Sadri et
  al.~\cite[Theorem~3]{BSCFR14} showed that for each odd $k$ and each factor
  $w$ of $\VTM$, the set of positions at which $w$ occurs in $\VTM$
  contains all congruence classes modulo $k$, including, in
  particular, $0$ modulo $k$.  This establishes the claim for all odd
  $k \geq 2$.

  If $k$ is even, write $k = 2^ak'$, where $k'$ is odd.  Suppose that
  $k' \geq 3$.  We have already seen that $\VTM$ contains an
  occurrence of a length $k'+1$ factor $0u0$ or $2u2$ at a position $i
  \equiv 0 \pmod{k'}$.  From the automaton generating $\VTM$, we see
  that if $v_i=0$ (resp.\ $v_i=2$), then $v_{2^ai}=0$
  (resp.\ $v_{2^ai}=2$), which establishes the claim for $k' \geq 3$.

  Finally, suppose that $k$ is a power of $2$.  Then the binary
  representations of $k$ and $2k$ have the form $10^\ell$ and
  $10^{\ell+1}$ respectively, for some $\ell \geq 1$.  From the
  automaton generating $\VTM$, we see that $v_k = v_{2k} = 2$, as
  required.  This completes the proof.
\end{proof}

This resolves Harju's first problem in the affirmative.  For the
third, we use the result of Currie~\cite{Cur13} (also referenced in
\cite{Har18}), who showed that for $k \geq 23$ there exists a
\emph{cyclic squarefree $k$-uniform morphism} on the ternary alphabet.
Let $h_k : \{0,1,2\}^* \to \{0,1,2\}^*$ denote such a morphism.  By
\emph{cyclic} we mean that if $\sigma$ denotes the morphism defined by
the cyclic permutation of the alphabet $0 \mapsto 1 \mapsto 2 \mapsto
0$, then for $i \in \{0,1,2\}$ we have $h_k(i) = \sigma^i(h_k(0))$.
By \emph{$k$-uniform} we mean that for $i \in \{0,1,2\}$ the images
$h_k(i)$ all have length $k$.  Finally, we say that the morphism $h_k$
is \emph{squarefree} if it maps squarefree words to squarefree words.
Without loss of generality, suppose that $h_k(i)$ begins with the
letter $i$.  Now if $w = (w_n)_{n \geq 0}$ is a squarefree word, then
$v = h_k(w)$ is also squarefree and moreover $(v_{kn})_{n \geq 0} =
(w_n)_{n \geq 0}$, as required.  This resolves Harju's third problem,
and furthermore, shows that solutions exist for every $k \geq 23$.

\end{document}